\documentclass[smallextended,envcountsame,numbook]{svjour3}
\usepackage{amsmath,amssymb,fullpage,xypic,mathrsfs,url}

\def\makeheadbox{\relax}

\newcommand{\reg}{{\operatorname{reg}}}
\newcommand{\IC}{\operatorname{IC}}
\newcommand{\bC}{\mathbf{C}}
\newcommand{\bA}{\mathbf{A}}
\newcommand{\Spec}{\operatorname{\mathsf{Spec}}}
\newcommand{\HP}{\operatorname{HP}}
\newcommand{\HH}{\operatorname{HH}}

\newcommand{\cO}{\mathcal{O}}
\newcommand{\cN}{\mathcal{N}}
\newcommand{\cB}{\mathcal{B}}
\newcommand{\bZ}{\mathbf{Z}}

\newcommand{\caD}{\mathcal{D}}

\newcommand{\iso}{\displaystyle \mathop{\to}^\sim}
\newcommand{\Lie}{\operatorname{Lie}}
\newcommand{\mfsl}{\mathfrak{sl}}
\newcommand{\mfgl}{\mathfrak{gl}}
\newcommand{\mfg}{\mathfrak{g}}
\newcommand{\mfh}{\mathfrak{h}}

\newcommand{\SL}{\operatorname{\mathsf{SL}}}

\newcommand{\Sym}{\operatorname{Sym}}

\newcommand{\Hom}{\operatorname{Hom}}

\newcommand{\Id}{\operatorname{Id}}

\newcommand{\rk}{\operatorname{rk}}
\newcommand{\pr}{\operatorname{pr}}

\newcommand{\Stab}{\operatorname{Stab}}

\newcommand{\Eu}{\operatorname{Eu}}
\newcommand{\gr}{\operatorname{gr}}
\newcommand{\ad}{\operatorname{ad}}
\newcommand{\into}{\hookrightarrow}
\newcommand{\ih}{\operatorname{ih}}
\newcommand{\IH}{\operatorname{IH}}
\newcommand{\Irrep}{\operatorname{Irrep}}

\AtBeginDocument{%
   \def\MR#1{}
}
\begin{document}
\author{Gwyn Bellamy 
and Travis Schedler}
\institute{Gwyn Bellamy
\at School of Mathematics and Statistics, University of Glasgow, University Gardens, Glasgow G12 8QW, United Kingdom, \\\email{gwyn.bellamy@glasgow.ac.uk}
\and
Travis Schedler
\at 
Department of Mathematics, Imperial College London, South Kensington Campus, London SW7 2AZ, United Kingdom, \\\email{trasched@gmail.com}
}

\title{Filtrations on Springer fiber cohomology and Kostka polynomials}
\date{\today}

\maketitle

\begin{abstract}
  We prove a conjecture which expresses the bigraded Poisson-de Rham
  homology of the nilpotent cone of a semisimple Lie algebra in terms
  of the generalized (one-variable) Kostka polynomials, via a formula
  suggested by Lusztig.  This allows us to construct a canonical
  family of filtrations on the flag variety cohomology, and hence on
  irreducible representations of the Weyl group, whose Hilbert series
  are given by the generalized Kostka polynomials.  We deduce
  consequences for the cohomology of all Springer fibers.  In
  particular, this computes the grading on the zeroth Poisson homology
  of all classical finite W-algebras, as well as the filtration on the
  zeroth Hochschild homology of all quantum finite W-algebras, and we
  generalize to all homology degrees.  As a consequence, we deduce a
  conjecture of Proudfoot on symplectic duality, relating in type A
  the Poisson homology of Slodowy slices to the intersection
  cohomology of nilpotent orbit closures. In the last section, we give an
  analogue of our main theorem in the setting of mirabolic D-modules.
\end{abstract}





\section{Introduction}
Let $\mfg$ be a semisimple complex Lie algebra, $\cN \subseteq \mfg^*$
the nilpotent cone (of elements whose coadjoint orbit is stable under
dilations), $W$ the Weyl group, and $G$ a simply-connected connected
complex Lie group with $\Lie G = \mfg$. The Springer correspondence
associates to every irreducible representation $\chi$ of $W$ a pair of
a nilpotent coadjoint orbit $\cO_\chi \subseteq \mfg^*$ and a local
system $L_\chi$ on $\cO_\chi$.  Let $\mathcal{B}$ be the flag variety
and $\rho: T^* \mathcal{B} \to \cN$ the Springer resolution.  Then the
cohomology of $T^* \mathcal{B}$, or equivalently of $\mathcal{B}$, is
endowed by the Springer correspondence with a $W$-action.  The graded
multiplicity space of each irreducible representation $\chi$ of $W$
has Hilbert series given by the generalized Kostka polynomial
$K_{\mfg,\chi}(t)$, which in the case of $\mfg=\mfsl_n$ is an ordinary
one-variable Kostka polynomial.  Precisely, we set
$K_{\mathfrak{g},\chi}(t) := \sum_{i \geq 0} t^i\dim \Hom_W(\chi,
H^{2\dim\cB - 2i}(\cB, \bC))$, where $\dim$ always refers to the
complex dimension.  Note that, as a graded $W$-module,
$H^*(\mathcal{B},\bC) \cong \Sym \mfh/ (\Sym \mfh)_+^W)$, 
putting $\mfh$
in degree two, with $((\Sym \mfh)_+^W)$
the ideal generated 
by the positive-degree $W$-invariant elements of $\Sym \mfh$.

By a theorem of \cite{ESwalg}, which was conjectured in \cite[Conjecture 1.21.(b)]{hp0weyl}
to generalize to arbitrary symplectic resolutions,
$H^*(T^*\mathcal{B}) \cong \HP_{\dim \cN-*}^{DR}(\cN)$, where the
latter is the Poisson-de Rham homology of $\cN$, defined in
\cite{ESdm}.  Briefly, the Poisson-de Rham homology of a Poisson
variety $Y$ is defined as the derived pushforward $\HP_i^{DR}(Y) :=
H^{-i} \pi_* M(Y)$ of the 
$\caD$-module $M(Y)$ on $Y$ to a
point, where $M(Y)$ is defined by the property $\Hom(M(Y), N) =
\Gamma_{\caD}(Y,N)^{H(Y)}$ for $H(Y)$ the Lie algebra of Hamiltonian vector fields on
$Y$, $N$ an arbitrary $\caD$-module on $Y$ (in the sense of
Kashiwara), and $\Gamma_{\caD}(Y,N)$ the global sections of $N$:
see Remark \ref{r:my-defn} below or \cite{ESdm,ES-survey} for
details. In the case of the nilpotent cone, the Poisson-de Rham
homology does not see the $W$-action, since $W$ does not act on $\cN$,
unlike on the cohomology of $T^*\mathcal{B}$.  On the other hand,
$\cN$ has a dilation action which endows $\HP_*^{DR}(\cN)$ with a
second grading, which is not seen in $H^*(T^*\mathcal{B})$.  It is
interesting to compute this grading. Moreover, this difference makes
it clear that the isomorphism of \cite{ESdm} cannot be canonical, and
it is interesting to correct this deficiency.

Lusztig suggested a simple formula for this bigrading
(\cite[Conjecture 8.1]{PS-pdrhhvnc}):
\begin{equation}\label{e:lusztig}
h(\HP_*^{DR}(\cN);x,y)=\sum_{\chi \in \Irrep(W)} K_{\mfg,\chi}(x^2)
  K_{\mfg,\chi}(y^{-2}).
\end{equation}
In this paper we prove this conjecture, in the following stronger form,
as a simple application of a theorem of Hotta and Kashiwara. Let $\sigma$ denote the sign representation of $W$.
\begin{theorem} \label{t:can-iso} There is a canonical isomorphism of
  bigraded vector spaces
 \[\HP_*^{DR}(\cN) \cong \Hom_W(\Sym \mfh/(\Sym \mfh)_+^W) \otimes \sigma, H^{2 \dim
    \cB-*}(T^*\mathcal{B})).
\]
\end{theorem}
Here the weight grading on the LHS corresponds to the grading on
$\Sym \mfh /((\Sym \mfh)^W_+)$ on the RHS (with $\mfh$ in degree two), and
the second grading is by the asterisk $*$. 
This isomorphism
accomplishes our goal of producing a canonical isomorphism. 
\begin{remark} Using the homotopy equivalence $T^*\mathcal{B} \simeq \mathcal{B}$ together with Poincar\'e duality for $\mathcal{B}$, we can rewrite the
theorem more simply as $\HP_*^{DR}(\cN) \cong \Hom_W(\Sym \mfh/((\Sym \mfh)_+^W), H^{*}(\mathcal{B}))$, but the way it is written is more natural; for
example, the aforementioned general conjecture  states
$\HP_*^{DR}(X) \cong H^{\dim X - *}(\widetilde X)$ for symplectic resolutions $\widetilde X \to X$.
\end{remark} 
 We go
further and produce canonical \emph{filtrations} on the cohomology of
the flag variety whose Hilbert series is given in \eqref{e:lusztig}:
\begin{theorem}\label{t:can-filt}
  For every element $\lambda \in \mfh_\reg^*$, there is a
  canonical associated filtration $\mathcal{F}_\lambda$ on
  $H^{2 \dim \cB-*}(T^*\mathcal{B})$ whose associated graded vector space is
  $\HP_*^{DR}(\cN)$. This is $W$-equivariant:
  $\mathcal{F}_{w(\lambda)} = w(\mathcal{F}_\lambda)$.
\end{theorem}
The filtration is compatible with the cohomological grading, hence the
associated graded vector space is bigraded.  As a result we obtain
canonical filtrations on irreducible representations of Weyl
groups. 
\begin{corollary}\label{c:irr-filt}
 To every $\lambda \in \mfh_\reg^*$, there is associated
a canonical filtration on every irreducible representation $\chi$ of $W$
whose associated graded vector space has Hilbert series $K_{\mfg,\chi}(y^{-2})$.
\end{corollary}
This recovers in particular the noncanonical
isomorphism predicted in \cite[Conjecture 8.1]{PS-pdrhhvnc}. 
\begin{example}
  Let $\mfg=\mfsl_n$ and let
  $\chi = \mfh^*  \cong \bC^{n-1}$ be the (dual)
  reflection representation.
  Consider $\chi$ to be (in coordinates)
  $\bC^n / \bC \cdot (1,1,\ldots,1)$.  Let $\lambda \in \chi$ be the
  image of $(a_1, \ldots, a_n) \in \bC^n$.  Then the resulting
  filtration on $\chi$, which we call the Vandermonde filtration, is
  given as follows: for every $0 \leq i \leq n$,
  $F^{2i-2\dim \mathcal{B}}(\chi)$ is the span of
  $(a_1^j, \ldots, a_n^j)$ for $1 \leq j \leq i$.
\end{example}
As we observe, the construction of Corollary \ref{c:irr-filt} actually
generalizes from Weyl groups to arbitrary complex reflection groups.
We will study the resulting filtrations in detail in future work.

We deduce many consequences and extensions of the above results to
Slodowy slices, $W$-algebras, and Springer fibers.  In more detail, let
$\phi \in \cN$ be any point.  Then one can consider the Slodowy slice
$S_\phi \cap \cN$ in $\cN$ to the coadjoint orbit
$\mathcal{O}_\phi := G \cdot \phi \subseteq \cN$.  (We recall its
construction in \S \ref{s:spr-fiber} below).  The ring of functions
$\cO(S_\phi \cap \cN)$ is also called a (centrally reduced) classical
W-algebra.
The above results allow us to deduce the grading on the zeroth Poisson
homology,
$$
\HP_0(\cO(S_\phi \cap \cN)):=\cO(S_\phi \cap \cN)/\{\cO(S_\phi \cap \cN),\cO(S_\phi \cap \cN)\} =
\HP_0^{DR}(S_\phi \cap \cN),
$$
%
as well as the filtration on the quantizations of $S_\phi \cap \cN$, which are
(centrally reduced) quantum W-algebras.  Geometrically, these
naturally assign to the top cohomology of each Springer fiber
$\rho^{-1}(\phi)$ a $\mfh_\reg^*$-family of filtrations whose
Hilbert series we compute.

As a consequence, when $\mfg=\mfsl_n$, and hence $Y=S_\phi \cap \mathcal{N}$ is
\emph{symplectically dual} to a corresponding coadjoint orbit
$Y^!  \subseteq \cN$, we prove a case of a conjecture of Proudfoot,
which states (for general symplectic dual cones $Y$ and $Y^!$) that
$\HP_0(\cO(Y)) \cong \IH^*(Y^!)$ as graded vector spaces, with
$\IH^*(Y^!)$ the intersection cohomology of $Y^!$.

We also give formulas for the higher Poisson-de Rham homology of
Slodowy slices, and for the zeroth Hochschild homology of their quantizations
by finite $W$-algebras.

\begin{remark}
  Note that \eqref{e:lusztig} implies that the weight grading on
  $\HP^{DR}_*(\mathcal{N})$ is nonpositive. This is somewhat unusual;
  for example whenever the zeroth Poisson homology of a conical
  Poisson variety is at least two-dimensional it will have (some)
  positive weights, as will happen for many Slodowy slices in
  the nilpotent cone, cf.~Corollary \ref{c:mc1}.  Note that, for
  varieties admitting a symplectic resolution for which
  \cite[Conjecture 1.3.(b)]{hp0weyl} holds, this condition that the
  zeroth Poisson homology has dimension at least two is equivalent to
  the statement that the fiber over the vertex in the symplectic
  resolution has multiple Lagrangian components (in particular, it is
  not irreducible).
\end{remark}


The proofs, given in \S \ref{s:dmod}, involve a study of the
Harish-Chandra $\caD$-module on $\cN$, following Hotta and Kashiwara
in \cite{HKihs}.  In Section \ref{ss:mirabolic}, we give a
generalization of our main result to the setting of mirabolic
$\caD$-modules, which computes the weakly equivariant structure of the
mirabolic Harish-Chandra $\caD$-module on $\mfgl_n \times \bC^n$,
defined in \cite{MirabolicCharacter}.

We begin the body of the paper in Section \ref{s:spr-fiber} with a
detailed statement of our results on the grading associated to the
cohomology of Springer fibers as well as to the Poisson and Hochschild
homology of W-algebras. The application to Proudfoot's conjecture on
symplectic duality is then given in Section \ref{s:proudfoot}. In the
remaining sections we prove our results using $\caD$-modules,
recalling first some of the necessary background.  In the last section, we
explain an alternative proof of Lusztig's formula using Hamiltonian
reduction. We then use this to generalize this result
 to the mirabolic setting, i.e.,
the setting of $\SL_n$-equivariant $\caD$-modules on
$\mathfrak{sl}_n \times \bC^n$.


\subsection{Acknowledgements}
The first author was partially supported by EPSRC grant EP/N005058/1.
The second author was partially supported by NSF grant DMS-1406553.
We would like to thank Pavel Etingof for useful discussions, and his
permission to use the results on filtrations (a special case of joint
work with the second author). The authors are grateful to the
University of Glasgow for the hospitality provided during the workshop
``Symplectic representation theory'', where part of this work was
done, and the second author to the 2015 Park City Mathematics
Institute and to the Max Planck Institute for Mathematics in
Bonn for their excellent working environments.

\section{Springer fibers and $W$-algebras}\label{s:spr-fiber}
Let $\phi \in \cN$.  We may then consider the Springer fiber
$\rho^{-1}(\phi) \subseteq T^* \cB$, which reduces to $\cB$ itself in
the case $\phi=0$.  

There is a beautiful construction of a transverse slice, called the
Kostant--Slodowy slice, which we denote $S_\phi$, to $\cO_\phi$ in
$\mfg^*$, which is an affine linear space defined as follows: Let
$\langle-,-\rangle$ be a non-degenerate invariant bilinear form on
$\mfg$ (e.g., the Killing form) and $\Phi: \mfg \to \mfg^*$ be the
resulting isomorphism.  Then $e:=\Phi^{-1}(\phi)$ is ad-nilpotent.
The Jacobson--Morozov theorem states that the element $e$ can be
extended (nonuniquely) to a so-called $\mathfrak{sl}_2$ triple
$(e,h,f)$ of elements of $\mfg$ satisfying the relations $[h,e]=2e,
[h,f]=-2f$, and $[e,f]=h$.  Then $S_\phi$ can be defined as $\Phi(e +
\ker(\ad f))$.  Moreover, Kazhdan defined a canonical contracting
$\bC^\times$ action on $S_\phi$ to $\phi$, by $\lambda \cdot \phi =
\lambda^{2-\ad(h)^*}(\phi)$; the induced grading on $\cO(S_\phi)$ is
called the Kazhdan grading.  Using this action, the resolution $\rho$
restricts to a $\bC^\times$-equivariant symplectic resolution $\rho:
\rho^{-1}(S_\phi \cap \cN) \to S_\phi \cap \cN$, which topologically
contracts to the fixed locus $\rho^{-1}(\phi) \to \{\phi\}$. Then, as
before, \cite[Conjecture 1.3.(b)]{hp0weyl} states in this case that
$H^*(\rho^{-1}(S_\phi \cap \cN)) \cong \HP^{DR}_{\dim S_\phi \cap \cN
  - *}(S_\phi \cap \cN)$, which was proved in \cite[Theorem
1.13]{ESwalg}.  Since the Poisson-de Rham homology is bigraded by
homological and Kazhdan gradings, this yields a bigrading on the
cohomology of the Springer fiber $H^*(\rho^{-1}(\phi))$, which was not
studied in \cite{ESwalg}. Our goal is to compute this grading.

For now, we describe the grading in top degree, $H^{\dim \rho^{-1}(\phi)}(\rho^{-1}(\phi))$
(see Corollary \ref{c:mc1} below for the general formula in terms of
intersection cohomology).  This has an explicit algebraic
interpretation in terms of $W$-algebras.  Namely, the finite
$W$-algebra $W_\phi$ is the coordinate ring of $S_\phi$, and its
central reduction $W_\phi^0$ is the coordinate ring of $S_\phi \cap
\cN$.  Let us recall their explicit algebraic description, along with
the Poisson structure, following \cite{GaGiqss} and \cite{Losfwa}.

Since $\ad(h)$ is semisimple, we get a decomposition $\mfg =
\bigoplus_i \mfg_i$ where $\mfg_i$ is the eigenspace of $\ad(h)$ of
eigenvalue $i$.  Equip $\mfg$ with the skew-symmetric form
$\omega_\phi(x,y) := \phi([x,y]) = \langle e, [x, y] \rangle$. This
restricts to a nondegenerate pairing on $\mfg_{-1}$.  Fix a Lagrangian
subspace $\mathfrak{l} \subseteq \mfg_{-1}$.  Then we define
\begin{equation}
  \mathfrak{m}_\phi := \mathfrak{l} \oplus \bigoplus_{i \leq -2} \mathfrak{g}_i.
\end{equation}
We also define the shift 
\begin{equation}
  \mathfrak{m}_\phi' := \{x-\phi(x) \mid x \in \mathfrak{m}_\phi\} 
  \subseteq \Sym \mfg.
\end{equation}
The $W$-algebra $W_\phi$ is then defined by
\begin{equation}
  W_\phi = (\Sym \mfg / \mathfrak{m}_\phi' \Sym \mfg)^{\mathfrak{m}_\phi}.
\end{equation}
In other words, this is the Hamiltonian reduction of $\mfg^*$ with
respect to the Lie algebra $\mathfrak{m}_\phi$ and its character
$\phi$.  By construction, $W_\phi$ is a Poisson algebra (with respect to
the Poisson bracket induced from $\Sym \mfg$).  In more detail, if $x
+\mathfrak{m}_\phi' \Sym \mfg \in W_\phi$, then $\{x,
\mathfrak{m}_\phi\} \subseteq \mathfrak{m}_\phi' \Sym \mfg$, and hence
$\{x, \mathfrak{m}_\phi' \Sym \mfg\} \subseteq \mathfrak{m}_\phi'
\Sym\mfg$.  Thus the Poisson bracket on $\Sym \mfg$ induces one on
$W_\phi$.

Given any central character $\eta: Z(\Sym \mfg) = (\Sym \mfg)^{\mfg}
\to \bC$, we can form the central reduction $W_\phi^\eta := W_\phi /
\ker(\eta) W_\phi$. Here $Z(\Sym \mfg)$ denotes the Poisson center of
$\Sym \mfg$. Then, $W_\phi$ and $W_\phi^0$ are the coordinate rings of
the Kostant--Slodowy slices:
\[
W_\phi = \cO(S_\phi), \quad W_\phi^0 = \cO(S_\phi \cap \cN).
\]
Note that, for general $\eta$, $\Spec W_\phi^\eta$ is a deformation of
$S_\phi \cap \cN$, namely the Kostant--Slodowy slice $S_\phi$
intersected with the closure of the regular coadjoint orbit on which
every $f \in (\Sym \mfg)^\mfg$ restricts to the constant function
$\eta(f)$.  Moreover $W_\phi^\eta$ is a filtered algebra with $\gr
W_\phi^{\eta} = W_\phi^0$.

Recall that the \emph{zeroth Poisson homology} of a Poisson algebra
$A$ is $\HP_0(A) := A/\{A,A\}$.  Then, as a corollary of our main
theorem, we compute the graded structure of the zeroth Poisson
homology of $W_\phi$ and $W_\phi^0$, as well as the filtered structure
of $W_\phi^\eta$. We will need to use the Springer correspondence,
which assigns (injectively) to each irreducible representation $\chi
\in \Irrep(W)$ the pair of a nilpotent coadjoint orbit $\cO_\chi
\subseteq \cN$ and an irreducible local system $L_\chi$ on $\cO_\chi$
(see the beginning of Section \ref{s:dmod} below for an explicit
definition of $(\cO_\chi,L_\chi)$).  Define the following polynomial
(cf.~\cite[(8.2)]{PS-pdrhhvnc}):
\begin{equation}
  P_{\phi}(y) :=  y^{\dim \cO_{\phi}}
  \sum_{\chi \in \Irrep(W) \mid \cO_\chi=\cO_\phi}
  \rk  L_\chi \cdot K_{\mathfrak{g},\chi}(y^{-2}).
\end{equation}
\begin{corollary}\label{c:hilb-hp0} The Hilbert series of
  $\HP_0(W^0_\phi)$, as well as of $\gr \HP_0(W^\eta_\phi)$ for all
  $\eta$, is $P_\phi(y)$.  Moreover, $\HP_0(W_\phi)$ is freely generated over
$(\Sym \mfg)^\mfg$ generated by a graded vector space of
  the same Hilbert series.
\end{corollary}
\label{r:stab} Note that $W_\phi$ and hence $\HP_0(W_\phi)$ inherit
actions of $\Stab_G(e,h,f)$ which commute with the dilation action and
hence preserve degrees; on $\HP_0(W_\phi)$ this action factors through
the finite group $\pi_0 \Stab_G(e,h,f)$ since the Lie algebra of
$\Stab_G(e,h,f) \subseteq G$ acts trivially.
As observed in \cite{ESwalg}, $\pi_0 \Stab_G(e,h,f) = \pi_0
\Stab_G(e)$ (since $\Stab_G(e,h,f)$ is the reductive part of
$\Stab_G(e)$), which is obviously equal to $\pi_0 \Stab_G(\phi)$, and
the isomorphism $H^*(\rho^{-1}(S_\phi \cap \cN)) \cong \HP^{DR}_{\dim
  S_\phi \cap \cN - *}(S_\phi \cap \cN)$ is compatible with the $\pi_0
\Stab_G(\phi)$-actions.

Moreover, $\pi_0 \Stab_G(\phi) = \pi_1 \cO_\phi$ (since we assumed $G$
simply-connected), and for $\chi \in \Irrep(W)$ such that
$\cO_\chi=\cO_\phi$, the local system $L_\chi$ is an irreducible
representation of $\pi_1 \cO_\phi$.  Thus, using the $\pi_0
\Stab_G(\phi)$ action, we can refine the corollary to yield the
following. Let $V_\chi^*$ be a graded vector space with Hilbert series
$K_{\mfg,\chi}(y^{-2})$.
\begin{corollary}\label{c:hilb-hp0-2}
As graded representations of $\pi_0 \Stab_G(\phi)=\pi_1 \cO_\phi$,
\[
\HP_0(W^0_\phi) \cong \bigoplus_{\chi \in \Irrep(W) \mid
  \cO_\chi=\cO_\phi} L_\chi \otimes V_\chi^*[-\dim \cO_\phi].
\]
\end{corollary}
Finally, for $W_\phi$ itself, Corollary \ref{c:hilb-hp0} yields
\begin{equation}
  h(\HP_0(W_\phi);y)=P_\phi(y) \prod_{i=1}^{r} (1-y^{2d_i})^{-1},
\end{equation}
where $r$ is the semisimple rank of $\mfg$ and $d_1, \ldots, d_r$ are
the degrees of the fundamental invariants (i.e., one-half the polynomial
degrees of generators of $(\Sym \mfg)^\mfg \cong (\Sym \mfh)^W$, for
$\mfh \subseteq \mfg$ a Cartan subalgebra with Weyl group $W$).
Similarly, Corollary \ref{c:hilb-hp0-2} implies we can write, as
graded representations of $\pi_0 \Stab_G(\phi)$,
\begin{equation}\label{e:wphi}
  \HP_0(W_\phi) \cong (\Sym \mfg)^\mfg \otimes \left( \bigoplus_{\chi \in
      \Irrep(W) \mid \cO_\chi=\cO_\phi} L_\chi \otimes V_\chi^*[-\dim
    \cO_\phi]\right).
\end{equation}

\subsection{Hochschild cohomology of quantum W-algebras}
Parallel to the previous corollaries, we can consider the quantum
analogue of $W_\phi$, defined as:
\begin{equation}\label{e:wphiq}
  W_\phi^q := (U \mfg / \mathfrak{m}_\phi' U \mfg)^{\mathfrak{m}_\phi},
  \quad W_\phi^{q,\eta} := W_\phi^q / \ker(\eta),
\end{equation}
where $\eta$ is a character of $Z(U\mfg) \cong (\Sym \mfg)^\mfg$.  By
\cite[Theorem 1.10.(i)]{ESwalg}, $\gr \HH_0(W^{q,\eta}_\phi) \cong
\HP_0(W^0_\phi)$, and it follows that $\gr \HH_0(W^q_\phi) \cong
\HP_0(W_\phi)$.  Thus, the following is an immediate consequence of
Corollary \ref{c:hilb-hp0} and its proof is omitted:
\begin{corollary}\label{c:quant}
  The Hilbert series of $\gr \HH_0(W^{q,\eta}_\phi)$, for all $\eta$,
  is $P_\phi(y)$.  Moreover, $\HH_0(W^q_\phi)$ is a free filtered module
  over $Z(U \mfg)$ generated by a filtered vector space whose
  associated graded vector space has Hilbert series $P_\phi(y)$.
\end{corollary}
Similarly, the associated graded vector space of $\HH_0(W^q_\phi)$ has Hilbert
series
\begin{equation}
  h(\gr \HH_0(W^q_\phi); y)=P_\phi(y) \prod_{i=1}^{r} (1-y^{2d_i})^{-1}.
\end{equation}
\begin{remark}
  As in Remark \ref{r:stab}, $\HH_0(W^{q,\eta}_\phi)$ and
  $\HH_0(W^q_\phi)$ are also representations of the finite group
  $\pi_0 \Stab_G(e,h,f)=\pi_0 \Stab_G(\phi)$, and Corollary
  \ref{c:hilb-hp0-2} and \eqref{e:wphi} carry over replacing the left
  hand sides by $\HH_0(W^{q,\eta}_\phi)$ and $\HH_0(W^q_\phi)$,
  respectively, now viewing $V_\chi^*$ as a filtered vector space.
\end{remark}

\subsection{Higher cohomology of the Springer fiber}
The next result describes the bigrading on the (full) cohomology of
the Springer fiber, which is analogous to the associated graded vector
space of $H^*(T^*\mathcal{B})$ appearing in Theorem
\ref{t:can-filt}. That is, we compute the Poisson-de Rham homology of
the Slodowy slices to the nilpotent cone.  We do not attempt here to
construct actual filtrations on the Springer fiber cohomology.

When $\overline {\cO_\chi} \supseteq \mathcal{O}_\phi$, we will
consider the varieties (called S3 varieties, after Slodowy,
Spaltenstein, and Springer),
$S_{\chi,\phi} := \overline{\mathcal{O}_\chi} \cap S_\phi$.  Let
$\ih_{\chi,\phi}(x) := h(\IH^*(S_{\chi,\phi},L_\chi|_{\mathcal{O}_\chi
  \cap S_\phi});x)$
be the intersection cohomology Poincar\'e polynomial of
$S_{\chi,\phi}$ equipped with the local system
$L_{\chi}|_{\mathcal{O}_\chi \cap S_\phi}$.  For example, in the case
$\mfg=\mfsl_n$, the $L_\chi$ are all trivial, and by \cite[Theorem
2]{Lus-gpsuc}, we have
$\ih_{\chi,\phi}(x)=x^{\dim S_{\chi,\phi}}K_{\lambda\mu}(x^{-2})$
where $\lambda$ and $\mu$ are the partitions of $n$ corresponding to
$\chi$ and $\phi$, respectively, and $K_{\lambda\mu}(x)$ is the
ordinary one-variable Kostka polynomial.
\begin{corollary}\label{c:mc1}
The bigraded Hilbert series of $\HP_*^{DR}(S_\phi \cap \cN)$ is 
$$
h(\HP_*^{DR}(S_\phi \cap \cN);x,y) = y^{\dim \cO_\phi} \sum_{\chi}
x^{\dim S_{\chi,\phi}} \ih_{\chi,\phi}(x^{-1}) K_{\mfg,\chi}(y^{-2}),
$$
where the sum is taken over all $\chi \in \Irrep(W)$ such that
$\overline{\mathcal{O}_\chi} \supseteq \mathcal{O}_\phi$.
\end{corollary}
In the case $\mfg=\mfsl_n$, we obtain (in slightly rewritten form) a
special case of a statement proved modulo Proudfoot's conjecture in
\cite[Proposition 6.1]{PS-pdrhhvnc} (we show in the next subsection
that the relevant case of Proudfoot's conjecture also follows from our
result). Let $X_{\lambda,\mu} := S_{\chi,\phi}$ where $\lambda$ is the
partition of $n$ corresponding to $\chi$ and $\mu$ is the partition of
$n$ corresponding to $\phi$. Then $X_{(n)\mu}=S_\phi \cap \cN$.
\begin{corollary}
The bigraded Hilbert series of $\HP_*^{DR}(X_{(n)\mu})$ equals $y^{2n_\mu} \sum_{\nu\geq \mu}
  K_{\nu\mu}(x^2)K_{\nu(1^n)}(y^{-2})$.
\end{corollary}
Here $n_\mu = \sum_i (i-1)\mu_i$ is the partition statistic of $\mu$,
which equals $\frac{1}{2}\dim \cO_\phi$, and $\leq$ is the dominance
ordering on partitions.

\section{Proudfoot's conjecture on symplectic duality}\label{s:proudfoot}
In \cite[3.4]{Pr12}, Proudfoot conjectured that, in the case that $X$
and $X^!$ are symplectic dual cones in the sense of
\cite[10.15]{BLPW-qcsr2} (with Poisson brackets of degree two), then
$\HP_0(\cO(X))\cong \IH^*(X^!)$ as graded vector spaces. We deduce
this now in a special case. Let $\mfg=\mfsl_r$ and let $\sigma$ be the
sign representation of the symmetric group $W=\mathfrak{S}_r$.  Let $\chi$ be an irreducible
representation of $W$ given by some partition $\lambda$ of $n$. Here
and following we denote the dual partition of $\lambda$ by
$\lambda^t$.  Let $\phi$ and $\phi'$ be nilpotent elements whose
Jordan blocks are given by the parts of $\lambda$ and $\lambda^t$,
respectively.  Let $X=S_\phi \cap \cN$ and $X^! =
\overline{\cO_{\phi'}}$.
The varieties $X$ and $X^!$ are symplectically dual (cf., e.g., \cite[\S 10.2.2]{BLPW-qcsr2}).
\begin{corollary} In the case above, $\HP_0(\cO(X))\cong \IH^*(X^!)$.
\end{corollary}
\begin{proof} This follows from Corollary \ref{c:hilb-hp0} by the
  proof of \cite[Proposition 8.9]{PS-pdrhhvnc}.  Here is an outline
  for the reader's convenience. By Corollary \ref{c:hilb-hp0},
  $h(\HP_0(\cO(X));y) = y^{\dim \cO_\chi} K_{\mfg,\chi}(y^{-2})$. On
  the other hand, it is well known that $h(\IH^*(X^!);x) = x^{\dim
    \cO_{\chi \otimes \sigma}}K_{\mfg,\chi \otimes \sigma}(x^{-2})$
  (see \eqref{e:ih-fla} below, noting that in this case all of the
  local systems $L_\chi$ are trivial).
 We thus need to establish the identity $t^{\dim
   \cO_\chi}K_{\mfg,\chi}(t^{-2})=t^{\dim \cO_{\chi \otimes
     \sigma}}K_{\mfg,\chi \otimes \sigma}(t^{-2})$, which is
 \cite[(8.3)]{PS-pdrhhvnc} (it follows from palindromicity of
 $K_{\mfg,\chi}(t)$, Poincar\'e duality for $H^*(\cB)$, and the
 dimension formula for $\cO_{\mfg,\chi})$.
\end{proof}
\begin{remark}
  By Corollary \ref{c:hilb-hp0-2}, we can also write a formula similar
  to the one above which holds for arbitrary type: as graded $\pi_0
  \Stab_G(\phi) = \pi_1 \cO_\phi$ representations,
\[
\HP_0(\cO(S_\phi \cap \cN)) \cong \bigoplus_{\chi \in \Irrep(W) \mid
  \cO_\chi=\cO_\phi} L_\chi \otimes \IH^*(\cO_\chi,L_\chi).
\]
\end{remark}

\section{$\caD$-module formulation and proof}\label{s:dmod}
\subsection{Recollections on $\caD$-modules on $\cN$}
We consider two weakly $\bC^\times$-equivariant $\caD$-modules
on $\cN$, studied in \cite{HKihs}.  The first is the pushforward,
$\rho_* \Omega_{T^*\cB}$, which is actually strongly equivariant
(since $\rho$ is equivariant).  In \cite{HKihs}, it is explained that
this $\caD$-module has a canonical $W$-action, and we obtain a
decomposition of $W$-equivariant $\caD$-modules,
\[
\rho_* \Omega_{T^*\cB} \cong \bigoplus_{\chi \in \Irrep(W)} \chi
\otimes \mathrm{IC}(\mathcal{O}_\chi,L_\chi),
\]
where $(\cO_\chi, L_\chi)$ is the pair of a nilpotent coadjoint orbit
$\cO_\chi$ and irreducible local system $L_\chi$ on $\cO_\chi$; this
is one way to define the Springer correspondence $\chi \mapsto
(\cO_\chi, L_\chi)$.

The second weakly $\bC^\times$-equivariant $\caD$-module on
$\cN$ is defined using the embedding $\cN \subseteq \mfg^*$. By definition (following Kashiwara),
$\caD$-modules on $\cN$ are canonically
identified with right $\caD$-modules on $\mathfrak{g}^*$ supported on
$\cN$, via the maps $M \mapsto i_* M$ and $N \mapsto i^! N$. Since
$\mathfrak{g}^*$ is smooth and affine, right $\caD$-modules there can
be defined as right modules over the ring $\caD(\mathfrak{g}^*)$ of
differential operators with polynomial coefficients.  We have the map
$\ad: \mathfrak{g} \to \caD(\mathfrak{g}^*)$, such that $\ad(x)$ is
the vector field acting by the adjoint action: precisely,
$\ad(x)(v)=[x,v]$ for $v \in \mathfrak{g} \subseteq \Sym \mfg =
\cO(\mfg^*)$, which extends uniquely to a derivation.

Then, we consider the $\caD$-module $M(\cN) := i^!\bigl(
(\ad(\mathfrak{g})+I(\cN))\cdot \caD(\mathfrak{g}^*) \setminus
\caD(\mathfrak{g}^*)\bigr)$. This is weakly $\bC^\times$-equivariant
with respect to the square of the dilation action on $\mfg^* \supseteq
\cN$, since $\ad(\mathfrak{g})$ and $I(\cN)$ are spanned by
homogeneous elements.
It is \emph{not} strongly equivariant in general, since the Euler
vector field need not be contained in $\caD(\mathfrak{g}^*) \cdot
(\ad(\mathfrak{g}) + I(\cN))$. (In fact, one can see that it is never
strongly equivariant, for example using our main result below.)
\begin{remark}\label{r:my-defn} The $\caD$-module $M(\cN)$, called $\mathcal{M}_0$ in
  \cite[\S 6]{HKihs}, can be identified with the $\caD$-module
  $M(\cN)$ defined in \cite{ESdm} using the Poisson bracket on
  $\cO(\cN)$. For the convenience of the reader, we recall the
  definition of the latter and the reason why it is the same as the
  $\caD$-module we define (although we will not actually need it for
  the arguments of this paper); for a more detailed treatment see
  \cite{ESdm,ES-survey}.  For a general affine Poisson variety $X$,
  i.e., the spectrum of a Poisson algebra $\cO(X)$, define
  $M(X) := H(X) \cdot \caD_X \setminus \caD_X$, where $H(X)$ is the
  Lie algebra of Hamiltonian vector fields on $X$, and $\caD_X$ is the
  standard $\caD$-module on $X$. The latter can be defined for any
  embedding $i: X \to V$ into a smooth variety $V$ as
  $\caD_X = i^! (I_X \cdot \caD(V) \setminus \caD(V))$, with $\caD(V)$
  the ring of differential operators on $V$ (viewed as a right module
  over itself) and $I_X \subseteq \cO(X)$ the ideal of $X$
  (Kashiwara's theorem implies this definition does not depend on the
  choice of $V$).  The $\caD$-module $\caD_X$ has the property
  $\Hom(\caD_X, N) = \Gamma_{\caD}(X,N)$, where
  $\Gamma_{\caD}(X,N) := \{f \in \Gamma(V, i_* N) \mid I_X \cdot f =
  0\}$
  is the functor of sections scheme-theoretically supported on $X$,
  which is independent of the choice of embedding as a consequence of
  the proof of Kashiwara's theorem. This property could also be used
  to define $\caD_X$. One then checks then that $H(X)$ acts as a Lie
  algebra on $\caD_X$ via its left multiplication on $\caD(V)$, which
  makes the definition given of $M(X)$ sensible. Explicitly,
  $M(X) = i^!\bigl( \caD(V) \cdot (I_X + \widetilde{H(X)}) \setminus
  \caD(V)\bigr)$,
  with $\widetilde{H(X)}$ the space of vector fields on $V$ which are
  parallel to $X$ and restrict there to elements of $H(X)$.  The
  $\caD$-module $M(X)$ has the defining property that
  $\Hom(M(X),N) = \Gamma_{\caD}(X,N)^{H(X)}$ for every $\caD$-module
  $N$ on $X$.

  In the case $X=\cN$ and $V=\mfg^*$, actually $V$ itself is
  Poisson. To identify $M(\cN)$ as in \cite{ESdm} with
  $(\ad(\mathfrak{g}) + I(\cN)) \cdot \caD(\mathfrak{g}^*) \setminus
  \caD(\mathfrak{g}^*)$,
  one can argue as follows.  In this case, $\cN \to \mfg^*$ is a
  Poisson embedding, i.e., the Poisson bracket on $\mfg^*$ preserves
  the ideal of $\cN$ and induces the bracket on $\cN$.  In particular,
  $H(\cN)=H(\mfg^*)|_{\cN}$.  Thus it suffices to show that
  $(\ad(\mathfrak{g}) + I(\cN)) \cdot \caD(\mfg^*) =
  (H(\mathfrak{g}^*)+I(\cN)) \cdot \caD(\mfg^*)$.
  In fact, more is true:
  $\ad(\mathfrak{g}) \cdot \caD(\mfg^*) = H(\mathfrak{g}^*) \cdot
  \caD(\mfg^*)$,
  and this holds for an arbitrary finite-dimensional Lie algebra
  $\mfg$.  To see this, first observe that
  $\ad(\mathfrak{g}) \subseteq H(\mathfrak{g}^*)$, so we only need to
  show that
  $H(\mathfrak{g}^*) \subseteq \ad(\mathfrak{g}) \cdot
  \caD(\mfg^*)$.
  Next, for every $f \in \cO(\mfg^*)$, we have
  $\xi_f = \sum_i f_i \ad(x_i)$ for some $f_i \in \cO(\mfg^*)$ and
  $x_i \in \mathfrak{g}$.  Then, in $\caD(\mfg^*)$, we have
  $\xi_f = \sum_i \ad(x_i) \cdot f_i - \sum_i \{x_i,f_i\}$.  Finally,
  one observes that $\sum_i \{x_i, f_i\} = 0$ since $\xi_f$ was
  Hamiltonian.  Thus
  $H(\mathfrak{g}^*) \subseteq \ad(\mfg) \cdot \caD(\mfg^*)$, as
  desired.
\end{remark}
\subsection{Hotta and Kashiwara's theorem}
 We recall the following result of Hotta and Kashiwara (the case
  $\lambda=0$ of \cite[Theorem 6.1]{HKihs}).  Let $\mathrm{Har}(\mfh^*) \subseteq
  \cO(\mfh^*)=\Sym \mfh$ be the subspace of harmonic polynomials, i.e.,
\[
\mathrm{Har}(\mfh^*) := \{f \in \cO(\mfh^*) \mid P(\partial_x)(f) = 0, \forall P \in
(\Sym \mfh^*)^W_+\},
\]
where we denote the embedding $\Sym \mfh^* \into \caD(\mfh^*)$ as
constant-coefficient differential operators by $P \mapsto
P(\partial_x)$, $(\Sym \mfh^*)^W$ is the $W$-invariant subalgebra, and
$(\Sym \mfh^*)^W_+$ is the augmentation ideal (of operators whose
constant term is zero).

Let $\sigma$ be the sign representation of $W$, placed in degree zero.
\begin{theorem}[Hotta and Kashiwara]\label{t:hk}
  There is a canonical isomorphism of weakly equivariant
  $\caD$-modules,
\begin{equation}
  M(\cN) \stackrel{\sim}{\longrightarrow} \Hom_W \left( \mathrm{Har}(\mfh^*) \otimes \sigma, \rho_* \Omega_{T^*\cB} \right).
\end{equation}
\end{theorem}
The above result follows from Theorem 5.2, Proposition 6.3.1, and
Theorem 6.1 of \cite{HKihs}, and the canonical isomorphism is
originally constructed as
$$
M(\cN)^F \stackrel{\sim}{\longrightarrow} \Hom_W \left( \mathrm{Har}(\mfh),
  (\rho_* \Omega_{T^*\cB})^F \otimes \sigma \right)
$$
where the superscript $F$ denotes the Fourier transform (there also
the $\caD$-modules are on $\mfg$ and $\mfh$ rather than $\mfg^*$ and
$\mfh^*$, so $\mathrm{Har}(\mfh^*)$ appears). But, for a weakly
$\mathbb{C}^{\times}$-equivariant $\caD$-module $M$ and graded vector
space $V$, the fact that $F(\Eu_{\mfg^*}) = - \Eu_{\mfg^*} - \dim
\mfg$ implies that $(M \otimes V)^F \cong M^F \otimes V^*$.
We recover the statement of the theorem.

\subsection{Proof of Theorem \ref{t:can-iso}}
Theorem \ref{t:can-iso} is an immediate application of Hotta and
Kashiwara's theorem.  Namely, pushing forward to a point,
$H^{-i} \pi_* M(\cN) = \HP_i^{DR}(\cN)$ and
\[
H^{-i} \pi_* \Hom_W \left( \mathrm{Har}(\mfh^*) \otimes \sigma, \rho_*
  \Omega_{T^*\cB} \right) = \Hom_W(\Sym \mfh/((\Sym \mfh)_+^W) \otimes
\sigma, H^{2 \dim \cB - i} (T^*\mathcal{B})).
\]

\subsection{Proof of Theorem \ref{t:can-filt}}
To prove Theorem \ref{t:can-filt}, we construct a
$\bC^\times$-equivariant vector bundle on $\mfh^*/W$.  Let
$\chi: \mfg^* \to \mfg^*/\!/G \iso \mfh^*/W$ be the coadjoint quotient
composed with the Chevalley isomorphism. Consider the
$(\cO(\mfg^*)^G, \caD(\mfg^*))$-bimodule
$\widetilde M := \ad(\mfg) \caD(\mfg^*) \setminus \caD(\mfg^*)$, equipped
with the $\cO(\mfg^*)^G$-structure by left multiplication by
$G$-invariant functions. This has the property that, for every
$\bar \lambda \in \mfg^*/\!/G$, the fiber
$\bar \lambda \otimes_{\cO(\mfg^*)^G} \widetilde M$ is nothing but
$M(\chi^{-1}(\bar \lambda))$.  Moreover, by \cite[Theorem
1.2]{LS-kerhomHC}, $\widetilde M$ is flat over $\cO(\mfg^*)^G$ (note that
the theorem there actually applies to the corresponding left
$\caD(\mfg^*)$-module
$\caD(\mfg^*) / \caD(\mfg^*) \ad(\mfg) = \widetilde M \otimes_{\cO_{\mfg^*}}
\Omega_{\mfg^*}$,
but its flatness is equivalent to that of $\widetilde M$).  Hence, this
tensor product is derived:
$\bar \lambda \otimes_{\cO(\mfg)^G} \widetilde M = H^0(\bar \lambda
\otimes^{L}_{\cO(\mfg)^G} \widetilde M)$.

We can now take the $\caD$-module pushforward to a point (under
$\pi: \mfg \to \Spec \bC$), to obtain $\pi_* \widetilde M$, a
quasi-coherent sheaf on $\mfg^*/\!/G \cong \mfh^*/W$ whose fibers are precisely the
Poisson-de Rham homologies:
\[
H^{-i} (\bar \lambda \otimes^L_{\cO(\mfg)^G} \pi_* \widetilde M) \cong
H^{-i} (\pi_* (\bar \lambda \otimes^L_{\cO(\mfg)^G} \widetilde M))
\cong H^{-i} (\pi_* \widetilde
M|_{\bar \lambda}) = \HP_{i}^{DR}(\chi^{-1}(\bar \lambda)).
\]
Here, the first isomorphism follows from associativity of derived tensor product, since $\pi_*(N) = N \otimes_{\caD(\mfg^*)} \cO(\mfg^*)$ for every right $\caD(\mfg^*)$-module $N$.
For $\lambda \in \mfh_\reg^*$, the variety
$\chi^{-1}(\bar \lambda)$ is symplectic and hence
$\HP_i^{DR}(\chi^{-1}(\bar \lambda)) = H^{\dim \chi^{-1}(\bar \lambda)
  - i}(\chi^{-1}(\bar \lambda))$ (by \cite[Example 2.6, Remark 2.14]{ESdm}).  

Next, consider the Grothendieck--Springer simultaneous resolution
$\widetilde{\mfg}^* \to \mfg^*$.  Let
$\mathfrak{n} := [\mathfrak{b}, \mathfrak{b}]$.  Recall that
$\widetilde{\mfg}^* := \{(\mathfrak{b},\phi) \in \mathcal{B} \times
\mfg^* \mid \phi \in \mathfrak{n}^\perp\}$.
Let $\pr_1: \widetilde{\mfg}^* \to \mathcal{B}$ and
$\pr_2: \widetilde{\mfg}^* \to \mfg^*$ be the two projections.  There
is also a canonical map
$\widetilde \pi: \widetilde{\mfg}^* \to \mfh^*$, given by
$\widetilde \pi(\mathfrak{b},\phi) =
\phi|_{\mathfrak{b}/\mathfrak{n}}$.
We can consider the composition
$\widetilde{\mfg}^* \to \mathcal{B} \times \mfh^* \to \mfh^*$, with
the first map $(\pr_1 \times \widetilde \pi)$.  The first map makes
$\widetilde{\mfg}^*$ a principal $\mathfrak{n}$-bundle over
$\mathcal{B} \times \mfh^*$.  The second map is the projection.  Put
together, we obtain a canonical isomorphism
$H^*(\widetilde \pi^{-1}(\lambda)) \cong H^*(\mathcal{B})$ for all
$\lambda$.  Furthermore, when $\lambda \in \mfh_\reg^*$, $\pr_2$
induces an isomorphism
$\displaystyle \widetilde \pi^{-1}(\lambda) \iso \chi^{-1}(\bar
\lambda) \subseteq \mfg^*$.
Therefore for regular $\lambda$ we obtain canonical isomorphisms
$H^{*}(\chi^{-1}(\bar \lambda)) \cong H^*(\mathcal{B})$ (note that
these depend on $\lambda \in \mfh_\reg^*$, and not just on the
$W$-orbit $\bar \lambda$).

Let $q: \mfh^* \to \mfh^*/W$ be the quotient. We can consider the
pullback
$q^* \pi_* \widetilde M = \cO(\mfh^*) \otimes_{\cO(\mfh^*)^W} \pi_*
\widetilde M$.
The previous paragraphs show that
$H^{-i}(q^* \pi_* \widetilde M)|_{\mfh_\reg^*}$ canonically
identifies with the Gauss--Manin system
$H^{2\dim \cB - i}(\widetilde \pi^{-1}(\lambda))$ over
$\mfh_\reg^*$, and that every fiber is canonically identified with
$H^{2 \dim \cB-i}(\mathcal{B})$.

\begin{lemma}\label{lem:equivb}
The sheaf $\pi_* \widetilde M$ is a finite rank $\bC^\times$-equivariant vector bundle on $\mfh^* / W$. 
\end{lemma}

\begin{proof}
First note that, by \cite[Theorem 1.1]{LS-kerhomHC},
$\widetilde M$ is a $\caD(\mfh^*)^W$-module. Let
$N := \Omega_{\mfh^*} \otimes_{\cO(\mfh^*)} \caD(\mfh^*)
\otimes_{\caD(\mfh^*)^W} \widetilde M$,
a right $\caD(\mfh^* \times \mfg^*)$-module.  We can furthermore
consider an alternative description of $N$ from \cite{HKihs}: Let
$\iota_1: \cO(\mfh^*)^W \to \cO(\mfg^*)^G$ and
$\iota_2: (\Sym \mfh^*)^W \to (\Sym \mfg^*)^G$ be the Chevalley
isomorphisms (considering $\Sym \mfh^* \subseteq \caD(\mfh^*)$ the
constant coefficient operators and similarly for $\mfg^*$). Define 
\begin{multline*}
N' = (\ad (\mfg) \caD(\mfh^* \times \mfg^*)
+ \{P - \iota_1(P) \mid P \in \cO(\mfh^*)^W\} \caD(\mfh^* \times \mfg^*)  \\
+ \{Q - \iota_2(Q) \mid Q \in (\Sym \mfh^*)^W
\}\caD(\mfh^* \times \mfg^*)) \setminus \caD(\mfh^* \times \mfg^*).
\end{multline*}
This is the right  corresponding to the left
 $\caD(\mfg^* \times \mfh^*)$-module
denoted $\mathcal{N}$ in \cite[Theorem 4.2]{HKihs}.
By
\cite[Theorem 4.2]{HKihs}, we know that $N'$ is a simple
holonomic $\caD(\mfh^* \times \mfg^*)$-module, and moreover,
\begin{equation}\label{e:hkihs}
N' \cong f_* \Omega_{\widetilde \mfg^*}
\end{equation}
where
$f = \widetilde \pi \times \pr_2: \widetilde{\mfg^*} \to \mfh^* \times
\mfg^*$.
By definition, there is a surjection $N' \to N$, and since $N'$ is
simple, $N \cong N'$.  Now letting
$\pi': \mfh^* \times \mfg^* \to \mfh^*$ be the first projection, we
have $\pi'_* N \cong \widetilde \pi_* \Omega_{\widetilde{\mfg^*}}$.
Decomposing $\widetilde \pi$ as the composition
$\widetilde{\mfg^*} \to \cB \times \mfh^* \to \mfh^*$, we obtain that
the right $\caD$-module $H^{-i}\pi'_* N$ corresponds (via the
right-left $\caD$-module correspondence followed by the
Riemann--Hilbert correspondence) to the trivial local system on
$\mfh^*$ with fibers $H^{2\dim \cB-i}(\cB)$.

Finally, by definition, we have
\begin{equation}\label{eq:NM}
  \pi'_* N = \Omega_{\mfh^*} \otimes_{\cO(\mfh^*)} \caD(\mfh^*)
  \otimes_{\caD(\mfh^*)^W} \pi_* \widetilde M.
\end{equation}
It is well-known that $\caD(\mfh^*)$ is an invertible
$(\caD(\mfh^*)^W, \caD(\mfh^*) \rtimes W)$-bimodule, inducing a Morita
equivalence between the two algebras. In particular, $\caD(\mfh^*)$ is
a projective $\caD(\mfh^*)^W$-module. Thus $\pi_* \widetilde M$ is
itself a $\caD(\mfh^*)^W$-module summand of $\pi'_* N$. The latter is
a graded finitely-generated projective $\cO(\mfh^*)$-module, and hence
also a finitely-generated projective $\cO(\mfh^*)^W$-module.  Thus, so
is $\pi_* \widetilde M$.  In other words, $\pi_* \widetilde M$ is a
$\bC^\times$-equivariant finite-rank vector bundle over
$\mfh^*/W$. This completes the proof of the lemma.
\end{proof}

Observe that, since the map $q$ is flat and $\bC^\times$-equivariant,
Lemma \ref{lem:equivb} implies that $q^* \pi_* \widetilde M$ is a
finite rank $\bC^\times$-equivariant vector bundle on $\mfh^*$. For
every $\lambda \in \mfh^*_\reg$, we can restrict
$q^* \pi_* \widetilde M$ to the line $\bC \cdot \lambda$, and we
obtain a $\bC^\times$-equivariant vector bundle on the line $\bA^1$,
$L_\lambda := (q^* \pi_* \widetilde M)|_{\bC \cdot \lambda}$.  But
finite rank $\bC^\times$-equivariant vector bundles on the line are
well-known to be the same thing as finite-dimensional filtered vector
spaces: given a finite-dimensional filtered vector space $F^\cdot V$,
one takes the associated $\bC[t]$-module
$L = \bigoplus_{i \in \bZ} F^{\leq i} V \cdot t^i$, and the opposite
direction is given by setting $V := L|_{t=1}$ and taking the
filtration by the image of weight spaces $\bigoplus_{j \leq i} L_j$.
Moreover, the fiber $L|_0$ at zero is the associated graded vector
space $\gr V$.  Therefore
$(q^* \pi_* \widetilde M)|_{\bC \cdot \lambda}$ is nothing but a
collection of filtrations on the underlying cohomologies
$H^{-i}(q^* \pi_* \widetilde M)|_{\lambda} \cong H^{2\dim
  \cB-i}(\cB)$.
This produces the desired filtrations on the cohomology of the flag
variety. The associated graded vector spaces are
$H^{-i} (q^* \pi_* \widetilde M)|_{0} = \HP^{DR}_i(\cN)$.

It remains to check the $W$-equivariance. First,
$(q^* \pi_* \widetilde M)|_{\lambda} = (q^* \pi_* \widetilde M)_{w(\lambda)}$
for all $\lambda$.  On the other hand, to obtain a filtration on the
flag variety cohomology, we identify
$(q^* \pi_* \widetilde M)|_{\mfh_\reg^*}$ with the Gauss--Manin system
of $\widetilde \pi^{-1}(\mfh_\reg^*) \to \mfh_\reg^*$. The latter
is $W$-equivariant, and the composition
$H^*(\widetilde \pi^{-1}(\lambda)) \cong H^*(\cB) \cong H^*(\widetilde
\pi^{-1}(w(\lambda))) \cong H^*(\widetilde \pi^{-1}(\lambda))$,
of applying the Gauss--Manin connection twice followed by the
isomorphism
$\widetilde \pi^{-1}(w(\lambda)) \cong \chi^{-1}(\bar \lambda) \cong
\widetilde \pi^{-1}(\lambda)$, is the action of $w$.

\subsection{Proof of Corollary \ref{c:irr-filt}}\label{s:proof-cor}
The proof consists of studying the relationship between the
right $\caD(\mfh^*)$-module $\pi'_*N$ and the left $\caD(\mfh^*)^W$-module
$\pi_* \widetilde M$.  Applying the Morita equivalence between $\caD(\mfh^*) \rtimes W$ and $\caD(\mfh^*)^W$ defined by $\caD(\mfh^*)$ to the expression (\ref{eq:NM}) for $\pi'_* N$ implies that 
\[
\pi_* \widetilde M = \caD(\mfh^*) \otimes_{\caD(\mfh^*) \rtimes W}
\Omega_{\mfh^*}^{-1} \otimes_{\cO(\mfh^*)} \pi'_* N =
(\Omega_{\mfh^*}^{-1} \otimes_{\cO(\mfh^*)} \pi'_* N)^W.
\]
More generally, we can consider the functor from weakly
$\bC^\times \times  W$-equivariant right $\caD(\mfh^*)$-modules to (weakly)
$\bC^\times$-equivariant left $\caD(\mfh^*)^W$-modules:
\[
T: \mathcal{F} \mapsto (\Omega_{\mfh^*}^{-1} \otimes_{\cO(\mfh^*)}
\mathcal{F})^W.
\]
For every irreducible representation $\chi$ of $W$, we can form the
strongly $\bC^\times \times W$-equivariant right
$\caD(\mfh^*)$-module, $\mathcal{F}_\chi := \Omega_{\mfh^*} \otimes \chi$
(which under the Riemann--Hilbert correspondence is the trivial local
system $\chi$ on $\mfh^*$ equipped with the $W$-linearization given by
the representation). 
We obtain the formula:
\[
T(\mathcal{F}_\chi) = (\chi \otimes \cO(\mfh^*))^W.
\]

This is a $\bC^\times$-equivariant vector bundle on $\mfh^*/W$.  For
any $\lambda \in \mfh^*_\reg$, we can restrict $T(\mathcal{F}_\chi)$
to the line $\bC \cdot \bar \lambda$ and get a
$\bC^\times$-equivariant vector bundle on the line, i.e., a
finite-dimensional filtered vector space
$T(\mathcal{F}_\chi)|_{\bar \lambda}$.  There is a canonical
isomorphism of vector spaces
\[
T(\mathcal{F}_\chi)|_{\bar \lambda} \iso \mathcal{F}_\chi|_{\lambda},  
\]
obtained from the composition 
\[
(\Omega_{\mfh^*}^{-1} \otimes_{\cO(\mfh^*)} \mathcal{F}_\chi)^W \into
\Omega_{\mfh^*}^{-1} \otimes_{\cO(\mfh^*)} \mathcal{F}_\chi
\mathop{\to}^{\text{vol}_{\mfh^*} \otimes -} \mathcal{F},
\]
by applying the restriction of the source to $\bar \lambda$ and the
target to $\lambda$.  This is an isomorphism because, for every
$W$-equivariant vector bundle $\mathcal{F}$ on $\mfh^*$ and
every $\lambda \in \mfh_\reg^*$, the fiber of $\mathcal{F}^W$ at
$\bar \lambda$ equals the fiber of $\mathcal{F}$ at $\lambda$.

Put together, for every $\lambda \in \mfh^*_\reg$, we obtain a
canonical filtration on the fiber $\mathcal{F}_\chi|_{\lambda}$.

Next, note that every weakly $\bC^\times \times W$-equivariant $\cO_{\mfh^*}$-coherent right $\caD_{\mfh^*}$-module is a direct sum of shifts $\mathcal{F}_\chi(k)$,
where the notation indicates a grading shift:  
$M(k) := M \otimes_\bC \bC_{-k}$, where $k \in \bZ$ and $\bC_k$ is the
representation of $\bC^\times$ in which $\gamma \in \bC^\times$ acts
by $\gamma^k \cdot \Id$. In the strongly equivariant case, $k=0$ for every
summand, i.e., these modules are direct sums of copies of $\mathcal{F}_\chi$.

Returning to $\pi'_* N$, we can write $\pi'_* N$ as a direct sum of
such $\mathcal{F}_\chi$ (with homological and weight shifts).  Let us determine the weight shifts.  Recall from \eqref{e:hkihs} and the preceding that $N$ is a simple holonomic right $\caD_{\mfh^* \times \mfg^*}$-module isomorphic to $f_* \Omega_{\widetilde \mfg^*}$. Equip the latter with the canonical strong 
$\bC^\times$-equivariant structure coming from this structure on 
$\Omega_{\widetilde \mfg^*}$. Then we can see from the proof of \cite[Theorem 4.2]{HKihs} that the isomorphism $N \to f_*\Omega_{\widetilde \mfg^*}$ is $W$-equivariant and sends the generator $[1] \in N$ to an element of degree $2 \dim \mathcal{B}$.  Thus, if we put $[1] \in N$ in degree zero, we obtain that $N \cong f_* \Omega_{\widetilde \mfg^*}(2 \dim \mathcal{B})$ as a weakly $\bC^\times \times W$-equivariant $\caD$-module.  Alternatively (but which amounts to the same proof), we 
observe that, since $N$ is simple, there is only a single value of the
shift, and then it must be $2 \dim \mathcal{B}$ in order to agree with Theorem \ref{t:can-iso}. Pushing forward, all weight shifts of the $\mathcal{F}_\chi$ appearing in $\pi'_* N$ are by $2 \dim \mathcal{B}$.

It follows that the filtration on each cohomology of the fiber at $\lambda$, 
$H^{-i}(\pi'_* N|_{\lambda}) = H^{2\dim \mathcal{B}-i}(\mathcal{B})$,
is a direct sum of copies of a single filtered vector space for each
irreducible representation $\chi$ of $W$, and the associated graded vector space $\gr \mathcal{F}_\chi|_{\lambda}$ of the latter is isomorphic to 
%
$(\chi \otimes 
\cO(\mfh^*))^W|_{0} \otimes \bC_{-2\dim \mathcal{B}}$. 
Recall that, by Chevalley's theorem, $\cO(\mfh^*)$ is a free
$\cO(\mfh^*)^W$-module, and in fact $\cO(\mfh^*) \cong (\cO(\mfh^*) / (\cO(\mfh^*)^W)_+) \otimes \cO(\mfh^*)^W$ as a $\cO(\mfh^*)^W \times W$-module.  
Put together,
 the Hilbert series of $\gr \mathcal{F}_\chi|_{\lambda}$ is
\begin{multline}
h((\chi \otimes \cO_{\mfh^*}(2\dim \cB))^W|_{0}; y)
= y^{-2\dim \cB} h(\Hom_W(\chi, \cO_{\mfh^*}/((\cO_{\mfh^*}^W)_+)); y) \\ = 
h(\Hom_W(\chi \otimes \sigma, \cO_{\mfh^*}/((\cO_{\mfh^*}^W)_+)); y^{-1})
=K_{\mfg,\chi}(y^{-2}),
\end{multline}
where the second equality is due to Poincar\'e duality for $H^*(\mathcal{B}) \cong \cO_{\mfh^*}/((\cO_{\mfh^*}^W)_+)$.


\subsection{The structure of $M(\cN)$}
The following statement was conjectured in \cite[Conjecture
8.1]{PS-pdrhhvnc}. As in \S \ref{s:spr-fiber}, let $V_\chi^*$ be a
weight-graded vector space with Hilbert series
$K_{\mfg,\chi}(y^{-2})$. Since $\rho_* \Omega_{T^*\mathcal{B}}$ is
strongly $\bC^\times$-equivariant and $\IC(\mathcal{O}_\chi,L_\chi)$
is a summand for all $\chi \in \Irrep(W)$, it follows that
$\IC(\mathcal{O}_\chi,L_\chi)$ admits the structure of a strongly
$\bC^\times$-equivariant $\caD$-module.  Let us equip it with this
structure.
\begin{theorem}\label{t:mt}
  There is an isomorphism of weakly equivariant $\caD$-modules,
\[
M(\cN) \cong \bigoplus_{\chi \in \Irrep(W)} V_\chi^* \otimes
\IC(\mathcal{O}_\chi,L_\chi).
\]
\end{theorem}
\begin{proof}
 The proof follows from 
 Hotta and Kashiwara's Theorem \ref{t:hk}. We need to observe 
that $\mathrm{Har}(\mfh^*)$ is canonically isomorphic (as
a graded $W$-representation) to $\Sym \mfh / ((\Sym \mfh)_+^W)$,
and thus to $H^*(\cB)$ as a graded $W$-representation.
So we get
\begin{multline}
  \Hom_W(\mathrm{Har}(\mfh^*) \otimes \sigma, \rho_* \Omega_{T^*\cB}) \cong
  \Hom_W(H^{*}(\cB) \otimes \sigma, \rho_* \Omega_{T^*\cB}) \\ \cong
  \bigoplus_{\chi \in \Irrep(W)} \Hom_W(\chi \otimes \sigma,
  H^*(\cB))^* \otimes \Hom_W(\chi, \rho_* \Omega_{T^*\cB}).
\end{multline}
Then, by definition, $\Hom_W(\chi, \rho_* \Omega_{T^*\cB}) \cong
\IC(\cO_{\chi},L_{\chi})$. By Poincar\'e duality, $\Hom_W(\chi \otimes
\sigma, H^*(\cB)) \cong \Hom_W(\chi, H^{2\dim \cB-*}(\cB))$, and the
Hilbert series of the latter is $K_{\mfg,\chi}(t^2)$, so its dual has
Hilbert series $K_{\mfg,\chi}(y^{-2})$. Thus, the RHS is isomorphic to
$\bigoplus_\chi V_\chi^* \otimes \IC(\cO_{\chi},L_\chi)$, as desired.
\end{proof}

\subsection{Proof of Lusztig's formula}\label{ss:proof-lusztig-conj}
Pushing $\rho_* \Omega_{T^*\cB} \cong \bigoplus_{\chi \in \Irrep(W)}
\chi \otimes \IC(\cO_{\chi},L_\chi)$ to a point, we get
\begin{equation}\label{e:ih-fla}
  x^{-\dim \cO_\chi}h(\IH^{*}(\cO_\chi,L_\chi);x) = x^{-2\dim \cB}h(\Hom_W(\chi,H^{*}(T^*\cB));x)= K_{\mfg,\chi}(x^{-2}).
\end{equation}
Thus, pushing the formula of Theorem \ref{t:mt} to a point, we obtain 
Lusztig's formula \eqref{e:lusztig}.

\section{Proofs of results on $W$-algebras and Slodowy slices}
\subsection{Proof of
 Corollaries
  \ref{c:hilb-hp0} and \ref{c:hilb-hp0-2}}
The first assertion of Corollary \ref{c:hilb-hp0} follows by
\cite[Remark 8.7]{PS-pdrhhvnc} (in more detail, it is a consequence of
Theorem \ref{t:mt} and \cite[Theorem 5.1]{PS-pdrhhvnc}). The assertion
for $W_\phi^\eta$ then follows immediately from \cite[Theorem
1.10.(ii)]{ESwalg}, stating that $\HP_0(W_\phi^\eta)$ is flat in
$\eta$. Hence, as explained in \cite[Theorem 1.10.(iii)]{ESwalg} and
its proof, $\HP_0(W_\phi)$ is a free graded module over
$\cO(\mfg)^\mfg$, so the assertion follows for $\HP_0(W_\phi)$ as
well. The same argument implies the refined statement, Corollary
\ref{c:hilb-hp0-2}.
\subsection{Proof of Corollary \ref{c:mc1}}
For Corollary \ref{c:mc1}, we need to compute $M(S_\phi \cap \cN)$ as
a weakly $\bC^\times$-equivariant $\caD$-module, with respect to its
dilation action with fixed point $\phi$.  We will use the notation
$N(k)$ for the shift of the weak $\bC^\times$-equivariant structure
on $N$ as defined in Section \ref{s:proof-cor}.
\begin{proposition}
  As weakly $\bC^\times$-equivariant $\caD$-modules,
\[
M(S_\phi \cap \cN)(\dim \cO_\phi) \cong \bigoplus_\chi V_\chi^* \otimes
\IC(S_{\chi,\phi},L_\chi|_{S_{\phi} \cap \cO_{\chi}}),
\]
the sum taken over $\chi \in \Irrep(W)$ such that $\overline{\cO_\chi}
\supseteq \cO_\phi$.
\end{proposition}
\begin{proof}
  Completing at $\phi$, we have, by the Darboux--Weinstein
  decomposition theorem, as formal Poisson schemes, $\hat \cN \cong
  \widehat{S_\phi \cap \cN} \times \hat \cO_{\phi}$.  But $\cO_{\phi}$
  is smooth so $\hat \cO_{\phi}$ is the completion of a symplectic
  vector space at the origin.  So, disregarding the equivariant
  structure, $M(\cN)|_{\hat \cN} \cong M(S_\phi \cap
  \cN)|_{\widehat{S_\phi \cap \cN}} \boxtimes \Omega_{\cO_\phi}|_{\hat
    \cO_{\phi}}$.  As a result, $M(S_\phi \cap \cN)$ is a direct sum
  of intermediate extensions of local systems on its leaves, which are
  $S_\phi \cap \cO_\chi$; the local systems which appear are
  $K_{S_\phi \cap \cO_\chi} := H^0 i_{S_\phi \cap \cO_\chi}^* M(S_\phi
  \cap \cN)$ where $i_{S_\phi \cap \cO_\chi}: S_\phi \cap \cO_\chi \to
  (S_\phi \cap \cN)$ is the inclusion.  This is even true together
  with the equivariant structure.  By \cite[Theorem 5.1]{PS-pdrhhvnc},
  the shift $K_{S_\phi \cap \cO_\chi}(-\dim (S_\phi \cap \cO_\chi))$ is
  the weakly equivariant local system
  described in \cite[\S 4.3]{ESdm} canonically given by
  attaching, to each $x \in S_\phi \cap \cO_\chi$, the fiber
  $\HP_0(\cO(S_x))$ where $S_x$ is the slice to $S_\phi \cap \cO_\chi$
  in $S_{\phi} \cap \cN$ at $x$.
  Precisely, the summand of $K_{S_\phi \cap
    \cO_\chi}$ which is weakly equivariant with respect to the
  character $m-\dim (S_\phi \cap \cO_\chi)$ of $\bC^\times$ is the
  local system attaching to each $x$ the weight $m$ subspace of
  $\HP_0(\cO(S_x))$.

  Note that $S_x$ is isomorphic to the same slice to $\cO_{\chi}$ in
  $\cN$ at $x$, so this is compatible with our previous notation.
  Passing back to $\cN$, we know again that $M(\cN)$ is a direct sum
  of intermediate extensions of local systems $L_{\cO_\psi}$ on its
  leaves $\cO_\psi$.  Again applying \cite[Theorem 5.1]{PS-pdrhhvnc},
  $K_{\cO_\psi}(-\dim \cO_{\psi})$ is again the canonical weakly
  equivariant local system attaching $\HP_0(\cO(S_x))$ to each point
  $x \in \cO_{\psi}$.

  By Theorem \ref{t:mt}, $K_{\cO_\psi} = \bigoplus_{\cO_\chi=\cO_\psi}
  V^*_\chi \otimes L_\chi$. Applying the two paragraphs above, we
  conclude that $K_{S_\phi \cap \cO_\chi} = \bigoplus_{\cO_\tau =
    \cO_\chi} V^*_\tau[\dim(S_\phi \cap \cO_\chi)-\dim \cO_\chi]
  \otimes L_\tau|_{S_\phi \cap \cO_{\chi}}$.  But $\dim (S_\phi \cap
  \cO_\chi)- \dim \cO_{\chi} = -\dim \cO_\phi$.  Summing over all
  $\chi$ yields the proposition.
\end{proof}
Now, pushing forward $M(S_\phi \cap \cN)$ to a point, we obtain
Corollary \ref{c:mc1}.

\section{Hamiltonian reduction and a mirabolic generalization}

\subsection{An alternative proof of Theorem \ref{t:mt}}\label{sec:alt}

We sketch an alternative way to complete the proof of Theorem
\ref{t:mt}, using the functor of Hamiltonian reduction. The details,
which are easily checked, are left to the interested reader. 
Let $G$ be  a connected complex Lie or algebraic group
with $\mfg = \Lie G$.
The Harish-Chandra homomorphism is a
surjective morphism $\delta: \caD (\mfg^*)^G \rightarrow
\caD(\mfh^*)^W$. This descends to an isomorphism $(\ad(\mfg)
\caD(\mfg^*))^G \setminus \caD (\mfg^*)^G \stackrel{\sim}{\longrightarrow}
\caD(\mfh^*)^W$. This allows one to define the functor of Hamiltonian
reduction $\widetilde{\mathbb{H}} : \mathrm{mod}$-$(\caD(\mfg^*),G) \rightarrow
\mathrm{mod}$-$\caD(\mfh^*)^W$, $\widetilde{\mathbb{H}}(M) = M^G$, from the category
of finitely generated, strongly $G$-equivariant right $\caD(\mfg^*)$-modules to the category
of finitely generated right $\caD(\mfh^*)^W$-modules.

We are interested in weakly $\bC^\times$-equivariant modules for the
squares of the dilation actions on $\mfg^*$ and $\mfh^*$.  Precisely, in
the former case we consider weakly $\bC^\times$-equivariant, strongly
$G$-equivariant right $\caD$-modules on $\mfg^*$, and in the latter case we
consider graded right $\caD(\mfh^*)^W$-modules.
For brevity we will call these weakly
equivariant modules on $\mfg^*$ or $\mfh^*$.  Let $\Eu_{\mfg^*}$ and
$\Eu_{\mfh^*}$ denote the Euler vector fields on $\mfg^*$ and $\mfh^*$, so
that the square of the dilation action in question is generated by
twice the Euler vector field.


One calculates that $\delta(\Eu_{\mfg^*}) = \Eu_{\mfh^*} - N$ where $N =
|R^+|$ and $R^+$ a choice of positive roots for $W$. Therefore, the
functor $\widetilde{\mathbb{H}}$ induces a functor between weakly
$\bC^\times$-equivariant modules such that $\widetilde{\mathbb{H}}(M)(2N)$ is a
strongly $\bC^\times$-equivariant $\caD(\mfh^*)^W$-module if $M$ is a
strongly $\bC^\times$-equivariant $\caD(\mfg^*)$-module, where as before $M(k) :=
M \otimes_\bC \bC_{-k}$.

We have
\begin{equation}\label{eq:iso1}
  \widetilde{\mathbb{H}}(M(\cN)) =  \cO(\mfh^*)^W_+  \caD(\mfh^*)^W \setminus \caD(\mfh^*)^W . 
\end{equation}
The rings $\caD(\mfh^*)^W$ and $\caD(\mfh^*) \rtimes \bC[W]$ are Morita
equivalent, compatibly with weak equivariance. Let
$\mathbb{H}$ be the composite functor $
\mathrm{mod}$-$(\caD(\mfg^*),G) \rightarrow \mathrm{mod}$-$\caD(\mfh^*)
\rtimes \bC [W]$. Then
$$
\mathbb{H}(M(\cN)) \simeq \cO(\mfh^*)^W_+ \setminus
\cO(\mfh^*) \otimes_{\cO(\mfh^*) \rtimes \bC[W]} \caD(\mfh^*) \rtimes \bC[W]
.
$$
The $\caD(\mfh^*) \rtimes \bC[W]$-modules $\Delta(\chi) := \chi
\otimes_{\cO(\mfh^*) \rtimes \bC[W]}(\caD(\mfh^*) \rtimes \bC[W])$,
for $\chi \in \mathrm{Irr} (W)$, are simple and
$\bC^\times$-equivariant, where $\cO(\mfh^*)_+$ acts
by zero on $\chi$ and
the group $\mathbf{C}^{\times}$ acts trivially on $\chi$. It has been shown in
\cite{HKihs} and \cite{LSssi}, respectively, that $M(\mathcal{N})$ is
semisimple and that $\mathbb{H}(\IC(\mathcal{O}_\chi,L_\chi)) \neq 0$
for all simple summands $\IC(\mathcal{O}_\chi,L_\chi)$ of
$M(\mathcal{N})$.\footnote{In \cite{LSssi} both statements are proved
  without explicitly using the proof of \cite{HKihs} of the first
  statement.} Since $\mathbb{H}$ is a quotient functor,
$\mathbb{H}(\IC(\mathcal{O}_\chi,L_\chi))$ is a simple $\caD(\mfh^*)
\rtimes \bC [W]$-module. Therefore
$\mathbb{H}(\IC(\mathcal{O}_\chi,L_\chi))(2N) \simeq \Delta(\lambda)$
for some $\lambda$. The key to proving Theorem \ref{t:mt} is to
identify $\lambda$. Recall that $\sigma$ denotes the sign
representation.

\begin{proposition}\label{prop:bij}
  As strongly $\bC^\times$-equivariant $\caD(\mfh^*) \rtimes
  \bC[W]$-modules,
  $\mathbb{H}(\IC(\mathcal{O}_\chi,L_\chi))(2N) \simeq
  \Delta(\chi \otimes \sigma)$.
\end{proposition}

\begin{proof}
  One can check that the functor of Hamiltonian reduction commutes
  with Fourier transform. Therefore, it suffices to show that
  $\mathbb{H}(\IC(\mathcal{O}_\chi,L_\chi)^F)(2N) \simeq \Delta^F(\chi
  \otimes \sigma)$, where $\Delta^F(\chi \otimes \sigma) := \chi
  \otimes_{\Sym \mfh^* \rtimes \bC[W]}\caD(\mfh^*) \rtimes
  \bC[W]$. Let $\nu : \widetilde{\mfg^*} \rightarrow \mfg^*$ be
  Grothendieck's simultaneous resolution. Then the direct image $\nu_*
  \Omega_{\widetilde{\mfg^*}}$ carries an action of $W$ such that
  $\nu_* \Omega_{\widetilde{\mfg^*}} \simeq \bigoplus_{\chi \in
    \mathrm{Irrep}(W)} \chi \otimes N_{\chi} $ as a
  $(W,\caD(\mfg^*))$-bimodule. By \cite[Theorem 5.2]{HKihs}, $N_{\chi
    \otimes \sigma} \simeq \IC(\mathcal{O}_\chi,L_\chi)^F$. Thus, it
  suffices to show that $\mathbb{H}(N_{\chi}) \simeq \Delta^F(\chi)$
  as $\caD(\mfh^*) \rtimes W$-modules. If $\mfg_{\mathrm{rs}}^*$
  denotes the open subset of $\mfg^*$ of elements whose image under
  the Killing form are regular semisimple in $\mfg$, then $N_{\chi}
  |_{\mfg^*_{\mathrm{rs}}}$ is a $G$-equivariant regular connection on
  $\mfg^*_{\mathrm{rs}}$. Thus, it descends to a regular connection on
  $\mfh_{\mathrm{reg}}^* / W$, whose image under the de Rham functor
  is a local system $R_{\chi}$. By the Springer correspondence,
  $R_{\chi}$ is isomorphic to the pullback of $\chi$ along the
  quotient map $\pi_1 (\mfh_{\mathrm{reg}}^* / W) \twoheadrightarrow
  W$. In more detail, the local system on $\mfh_{\mathrm{reg}}^* / W$
  corresponding to $\nu_* \Omega_{\widetilde{\mfg^*}}$ is the pullback
  of the regular representation under $\pi_1(\mfh_{\mathrm{reg}}^*/W)
  \twoheadrightarrow W$, which is how the action of $W$ is defined on
  $\nu_* \Omega_{\widetilde{\mfg^*}}$ in the first place, and from
  this the statement follows.

  Similarly, both $\mathbb{H}(N_{\chi})$ and $\Delta^F(\chi)$ restrict
  to $W$-equivariant regular connections on
  $\mfh_{\mathrm{reg}}^*$. Therefore, if $e \in \bC [W]$ is the trivial
  idempotent, $e(\mathbb{H}(N_{\chi})
  |_{\mfh_{\mathrm{reg}}^*})$ and $e( \Delta^F(\chi)
  |_{\mfh_{\mathrm{reg}}^*})$ are regular connections on
  $\mfh_{\mathrm{reg}}^* / W$. It is clear that the de Rham functor
  applied to $e( \Delta^F(\chi) |_{\mfh_{\mathrm{reg}}^*})$ gives the
  local system $R_{\chi}$. Thus, it suffices to show that the de Rham
  functor applied to $e(\mathbb{H}(N_{\chi}) |_{\mfh_{\mathrm{reg}}^*})$
  equals $R_{\chi}$.  Equivalently, if $\mathbb{H}^{\perp}$
  is the right adjoint to $\mathbb{H}$, then it suffices
  to show that $\mathbb{H}^{\perp}(\Delta^F(\chi) )
  |_{\mfg^*_{\mathrm{rs}}} \simeq N_{\chi} |_{\mfg^*_{\mathrm{rs}}}$ as
  $G$-equivariant regular connections. This is proved in the same way
  as \cite[Proposition 5.5.1]{MicrolocalKZ}, using all of the preceding.
\end{proof}

The graded multiplicity of $\chi$ in $\cO(\mfh^*)/\cO(\mfh^*)^W_+$ is
given by the graded vector space $V_{\chi}$ with Hilbert series
$K_{\mfg,\chi}(y^2)$ and
\begin{equation}\label{eq:iso2}
  \cO(\mfh^*)^W_+ \setminus \cO(\mfh^*) \otimes_{\cO(\mfh^*) \rtimes \bC[W] } \caD(\mfh^*)\rtimes \bC[W] = \bigoplus_{\chi \in \mathrm{Irrep} (W)}  V_{\chi}  \otimes\Delta(\chi)
\end{equation}
as weakly equivariant modules. Therefore, if $U_{\chi}$ is the graded
multiplicity of $\IC(\mathcal{O}_\chi,L_\chi)$ in $M(\cN)$, equations
(\ref{eq:iso1}) and (\ref{eq:iso2}) imply that $U_{\chi} \simeq
V_{\chi \otimes \sigma}(2N)$. The socle of the Gorenstein ring
$\cO(\mfh^*)/\cO(\mfh^*)^W_+$ is in degree $2N$, where it carries a copy
of the sign representation. Therefore, $\cO(\mfh^*)/\cO(\mfh^*)^W_+$ is,
up to a shift by $-2N$ and twisting by $\sigma$, isomorphic as a
graded $W$-module to its dual. This implies that $V_{\chi \otimes
  \sigma}(2N) \simeq V_{\chi}^*$ as graded vector spaces, and hence
$h(U_{\chi};y^2) = h(V_{\chi}^*;y^2) = K_{\mfg,\chi}(y^{-2})$.

\subsection{The mirabolic case}\label{ss:mirabolic}

In this section only, let $\mfg = \mathfrak{gl}(V)$ for some
$n$-dimensional vector space $V$ and $G = GL(V)$. Then $G$ acts
diagonally on $\mfg^* \times V$.  The group $\mathbb{C}^{\times}$ also
acts by dilations \textit{along the $\mfg^*$ factor}, i.e., $\alpha \cdot
(X, v) = (\alpha^{-2} X, v)$.  Fix $c \in \mathbf{C}$, thought of as a
character $X \mapsto c \mathrm{Tr}(X)$ of $\mfg$. If $\mu : \mfg
\rightarrow \caD$ is the quantum moment map, then a right $\caD$-module $M$
is said to be $(G,c)$-monodromic if $\mu_M(X) + \mu(X) = c \cdot
\Id_M$, where $\mu_M$ is the differential of the $G$ action on $M$ (we
have a sum rather than a difference because we are considering right
$\caD$-modules). Similarly, we say that $M$ is
$(\bC^{\times},d)$-monodromic if $2 \Eu_M + 2 \Eu_{\mfg^*} = d \cdot
\Id_M$, where $\Eu_M$ is the differential of the $\bC^{\times}$ action
on $M$. These notions can be defined sheaf-theoretically; see
\cite[Definition 2.3.2]{mirabolicHam}.  Note that monodromic
$\caD$-modules are automatically weakly equivariant.

The analogue of the $\caD$-module $M(\cN)$ in this case is the
mirabolic Harish-Chandra $\caD$-module (or rather its Fourier
transform), $M_c(\cN \times V)$, as introduced in
\cite{MirabolicCharacter}. As a $(G,c)$-monodromic, weakly equivariant
$\caD$-module on $\mfg \times V$,
$$
M_c(\cN \times V) = (\cO(\mfg^*)_+^G \ \caD + \mu_c(\mfg) \ \caD)
\setminus \caD.
$$
where $\mu_c(\mfg) = \{ \mu(X) - c \mathrm{Tr}(X) \ | \ X \in \mfg
\}$. This $\caD$-module can also be viewed as the natural module
associated to the $c$-twisted action of $\mfg$ on $\cN \times V$, as
in \cite[Remark 2.17]{ES-dmlv}.

The space $\cN \times V$ consists of finitely many $G$-orbits. The
orbits with fundamental group $\mathbf{Z}$ are naturally labeled by
partitions $\lambda \in \mathcal{P}_n$ of $n$. Each of these orbits
$\cO_{\lambda}$ admits a unique irreducible (one-dimensional)
$(G,c)$-monodromic local system $L_{\lambda,c}$. For each $\lambda \in
\mathcal{P}_n$, set $c_{\lambda} := c (n_{\lambda^t} - n_{\lambda})$,
where $\lambda^t$ the dual of $\lambda$; recall that $n_{\lambda} =
\sum_i (i-1) \lambda_i$ is the partition statistic.

We will say that $c$ is \emph{generic} if $c \notin \{\frac{r}{m} \mid 1 \leq m \leq n, \gcd(r,m)=1\}$.

\begin{theorem}\label{thm:mirabolic}
  Assume $c$ is generic.
  There exists a permutation $\tau$ of $\mathcal{P}_n$ such
  that $c_{\tau(\lambda)} = c_{\lambda}$ and an isomorphism of
  $(G,c)$-monodromic, weakly $\mathbf{C}^{\times}$-equivariant
  $\caD$-modules
$$
M_c(\cN \times V) \simeq \bigoplus_{\lambda \vdash n}
V_{\tau(\lambda)}^* \otimes \mathrm{IC}(\cO_{\lambda},L_{\lambda,c}),
$$
where $V_\lambda^*$ is a weight-graded vector space with Hilbert
series $K_{\mfg,\lambda}(y^{-2})$.
\end{theorem}
  
The basic idea behind the proof of this theorem is essentially the
same as the one outlined in section \ref{sec:alt}. Again, there is a
functor of Hamiltonian reduction,
$\mathbb{H}_c : \mathscr{C}_c \rightarrow \cO_{-c}$, where
$\mathscr{C}_c$ is the category of $(G,c)$-monodromic $\caD$-modules
supported on $\cN \times V$ and $ \cO_{c}$ denotes category $\cO$ for
the rational Cherednik algebra $H_c(S_n)$; see
\cite{AlmostCommutingVariety} for details. In the case where $c$ is
generic, category $\cO_c$ for the rational Cherednik algebra is
semisimple and the functor of Hamiltonian reduction induces an
equivalence between $\mathscr{C}_c$ and category $\cO_{-c}$ (see
\cite[Proposition 9.13]{BB-enhnil-quiver}). The argument of section
\ref{sec:alt} is applicable in this setting, though it is more
involved since the $\caD$-modules
$\mathrm{IC}(\cO_{\lambda},L_{\lambda,c})$ are
$(\mathbf{C}^{\times},c_{\lambda})$-monodromic, unlike the classical
setting where they can be endowed with a
$\mathbf{C}^{\times}$-equivariant structure. Moreover, the reason for
the occurrence of the permutation $\tau$ is the fact that the analogue
of Proposition \ref{prop:bij} is missing in this context. This is
because the key to the proof of Proposition \ref{prop:bij} is the
geometric construction of the simple modules $N_{\chi}$. Since $c$ is
assumed to be generic, there is no analogous construction for the
corresponding simple mirabolic modules.  It is an interesting question
if $\tau$ is the identity, and if not, it would be interesting to
compute it (there would not seem to be an obvious nontrivial
permutation satisfying $c_{\tau(\lambda)}=c_\lambda$).

As shown in \cite{mirabolicHam}, the case where $c$ is not generic is
much more interesting (in particular, there the category of mirabolic
sheaves need not be semisimple, and we expect $M_c(\cN \times V)$ not
to be semisimple when the category is not). We will return to this in
future work, where details of the proof of Theorem \ref{thm:mirabolic}
will also be given.

\bibliographystyle{amsalpha}
\bibliography{master}
\end{document}